\documentclass[preprint,12pt,authoryear]{elsarticle}

\usepackage{amssymb}

\usepackage{amsthm}

\usepackage{mathtools}
\newtheorem{thm}{Theorem}
\usepackage{multirow}
\usepackage{amsmath,amsfonts,amssymb,amsthm}
\usepackage{caption}
\usepackage{subcaption}

\journal{Nuclear Physics B}

\begin{document}

\begin{frontmatter}



\tnotetext[label1]{corresponding author.}
 \author{Amenah AL-Najafi\corref{cor1}\fnref{label1}}
 \ead{amenah@math.u-szeged.hu}
 \author{L\'aszl\'o Viharos\corref{cor1}\fnref{label1}}
 \ead{viharos@math.u-szeged.hu}
 

\title{Regression estimators for the tail index}

 \address[label1]{Bolyai Institute, University of Szeged, Aradi v\'ertan\'uk tere 1, Szeged, H-6720, Hungary}



\begin{abstract}
We propose a class of weighted least squares estimators for the tail index of a distribution function with a regularly varying upper tail. Our approach is based on the method developed by \cite{Holan2010} for the Parzen tail index. Asymptotic normality of the estimators is proved. Through a simulation study, these and earlier estimators are compared in the Pareto and Hall models using the mean squared error as criterion. The results show that the weighted least squares estimator is better than the other estimators investigated.

\end{abstract}

\begin{keyword}
tail index, weighted least squares estimators, Pareto model, quantile process.

\end{keyword}

\end{frontmatter}

\section{Introduction and main result}
\label{}

Let $X_1, X_2,\ldots$ be independent random variables with a common right-continuous distribution function $F$, and for each $n \in \mathbb N$, let $X_{1,n}\le \cdots \le X_{n,n}$ denote the order statistics pertaining to the sample $X_1, \ldots, X_n$. 
Let ${\cal R}_\alpha$ be the class of all distribution functions $F$ such that $1-F$ is regularly varying at infinity with index $-1/\alpha$, that is, 
\begin{equation*}
1-F(x)=x^{-1/\alpha}\ell(x),\quad 1 < x<\infty,
\end{equation*}
where $\ell$ is some positive function on the half line $[1,\infty)$, slowly varying at infinity and $\alpha > 0$ is a fixed unknown parameter to be estimated.
Introducing the quantile function $Q$ of $F$ defined as
\[Q(s):=\inf\left\lbrace x:F(x)\ge s\right\rbrace, \quad0<s\le1,\quad Q(0):=Q(0+),\]
it is well known that $F\in{\cal R}_\alpha$ if and only for some function $L$ slowly varying at zero,
\begin{equation}\label{quantile}
Q(1-s)=s^{-\alpha}L(s),\quad 0< s<1.
\end{equation}

Several estimators exist for the tail index $\alpha$ among which Hill's estimator
is the most classical. \cite{Hill1975} proposed the following estimator for the tail index $\alpha$:
\[ \widehat\alpha_n^{(H)}=\frac1{k_n}\sum_{j=1}^{k_n}\log X_{n-j+1,n}-\log X_{n-k_n,n},\]
where the $k_n$ are positive integers, which in theoretical
asymptotic considerations will satisfy the conditions
\[1\le k_n < n, \quad\ k_n\to\infty\quad\ {\rm and}\quad\ k_n/n\to0 \quad{\rm as}\ n\to\infty.\]
The asymptotic normality of $\widehat\alpha_n^{(H)}$ was first considered by \cite{Hall1982} in the following submodel of ${\cal R}_\alpha$:
\[1-F(x)=x^{-1/\alpha}C_1[1 + C_2x^{-\beta/\alpha}\{1 + o(1)\}],\quad{\rm as}\ x\to\infty,\]
for some constants $C_1 > 0$ and $C_2\not=0$. This is equivalent to
\begin{equation}\label{Hallm}
Q(1-s)=s^{-\alpha}D_1[1+D_2s^{\beta}\left\lbrace 1+o(1)\right\rbrace ],\quad s\rightarrow 0,
\end{equation}
where $D_1 = C_1^{\alpha}$ and $D_2 = C_2/C_1^{\beta}$.

Another estimators were proposed by \cite{Pickands1975}, \cite{Dekkers1989}, to name a few.

Assuming that $F$ is absolutely continuous with density function $f$, \cite{Parzen2004} studied the following alternative model for the right tail of the distribution:
\[fQ(s):=f(Q(s))=(1-s)^\nu L_1(1-s),\quad s\in(1/2,1],\]
where $\nu>0$ is a finite constant and $L_1$ is slowly varying at zero.
The parameter $\nu$ is called the Parzen tail index of the density-quantile function $fQ(\cdot)$.

Based on an orthogonal series expansion for $L_1$, \cite{Holan2010} introduced a regression estimator for the Parzen tail index using ordinary least squares. \cite{Viharos2020} obtained a more general class of estimators for $\nu$ using weighted least squares.
We adopt this method to estimate the classical tail index $\alpha$. Following the idea of \cite{Holan2010}, we assume that the slowly varying function $L$ in (\ref{quantile}) admits the truncated orthogonal series expansion 
\begin{equation*}
L(s)=\exp\left\lbrace\theta_0+2\sum_{k=1}^{p}\theta_k\cos(2\pi ks)\right\rbrace, \label{oseries}
\end{equation*}
where $p>0$ is a fixed unknown integer, and $\theta_0,\ldots,\theta_p$ are unknown parameters. 
It follows that
\begin{equation}
\log Q(1-s)=-\alpha\log s+\theta_0+2\sum_{k=1}^{p}\theta_k\cos(2\pi ks).\label{Qrepr}
\end{equation}
Let $Q_n$ be the empirical quantile function defined as
\begin{equation*}
Q_n(s)=X_{k,n} \qquad \text{if} \quad \frac{k-1}n<s\le\frac kn, \quad k=1,2,\ldots, n. 
\end{equation*}
Based on the representation (\ref{Qrepr}), we obtain the regression equations
\[\log{Q_n}(1-s_j)=-\alpha\log s_j+\theta_0+2\sum_{k=1}^{\widetilde p}\theta_k\cos(2\pi ks_j)+\varepsilon(s_j),\]
where $\varepsilon(s)=\log (Q_n(1-s)/Q(1-s))$ is the residual process, $s_j=j/n,\ j=\lceil na\rceil, \ldots, \lfloor nb\rfloor$, $a<b$ are fixed constants taken from the interval (0,1), $\widetilde p>p$ is chosen by the statistician and $\theta_k=0$ for $k>p$. We propose a class of estimators for $\alpha$ using weighted least squares. We choose some nonnegative weights of the form $w_{j,n}=R(j/n)$ with some weight function $R$. Set $y_j:=\log Q_n(1-s_j),$ 
\[y:=(y_{\lceil na\rceil},\ldots,y_{\lfloor nb\rfloor})',\]
\[W:=\textrm{diag}(w_{\lceil na\rceil,n}, \ldots, w_{\lfloor nb\rfloor,n}),\]
and let $X:=[G^*, G_0, 2G_1,\ldots, 2G_{\widetilde p}]$, where
\[G^*=\big(-\log(s_{\lceil na\rceil}), \ldots, -\log(s_{\lfloor nb\rfloor})\big)',\]
\[G_k=\big(\cos(2\pi ks_{\lceil na\rceil}), \ldots, \cos(2\pi ks_{\lfloor nb\rfloor})\big)', \quad k=0, \ldots,\widetilde p.\]
Set $\beta_{\widetilde p}:=(\alpha, \theta_0, \theta_1, \ldots ,\theta_{\widetilde p})'$.
By minimizing the weighted sum of squares
\[\sum_{\lceil na\rceil}^{\lfloor nb\rfloor}w_{j,n}\big(y_j+\alpha\log s_j-\theta_0-2\sum_{k=1}^{\widetilde p}\theta_k\cos(2\pi ks_j)\big)^2,\]
we obtain the following estimator of $\beta_{\widetilde p}$:
\[\widehat{\beta}_{\widetilde p}=(X'WX)^{-1}X'Wy.\]
Then the weighted least squares estimator of $\alpha$ can be written in the form 
\[\widehat\alpha_n^{(W)}:=e_1'\widehat\beta_{\widetilde p}=e_1'(X'WX)^{-1}X'Wy,\]
where $e_1$ is the $\widetilde p+2$ dimensional vector defined as $e_1=(1,0,0,\ldots,0)'$. 

\bigskip
We assume the following conditions on the underlying distribution:

\bigskip\noindent
$(Q_{1})$ The distribution function $F$ is continuous and twice differentiable on $(a^{*},b^{*})$, where $a^{*}=\sup\left\lbrace x:F(x)=0\right\rbrace $, $b^{*}=\inf\left\lbrace x:F(x)=1\right\rbrace, -\infty\leq a^{*}<b^{*}\leq \infty$ and $f(x):=F'(x)\neq0$ on $(a^{*},b^{*}).$\\
$(Q_{2})\ \sup_{a^{*}<x< b^{*}}F(x)(1-F(x))|f'(x)/f^{2}(x)|<\infty.$\\
$(Q_{3})\ \sup_{1-b\le s\le 1-a}1/|Q(s)|<\infty$, $\sup_{1-b\le s\le 1-a}1/fQ(s)<\infty$ and
\hfill\break $\sup_{1-b\le s\le 1-a}1/|fQ(s)Q(s)|<\infty.$\\ 

We will show that the limit matrix $M(a,b,R):=\lim_{n\to\infty}n^{-1}X'WX$ exists (see the proof of Theorem \ref{mt} in Chapter 3). Let $(v^*, v_0,\ldots, v_{\widetilde p})$ be the first row of $M(a,b,R)^{-1}$, and set $G_R(u):=R(u)\big(-v^*\log u+v_0+2\sum_{k=1}^{\widetilde p}v_k\cos(2\pi ku)\big)$ for $u\in(0,1)$.\\

Moreover, we suppose the following conditions:

\bigskip\noindent
$(R)$ The weight function $R$ is nonnegative and Riemann integrable on $[a,b].$\\
$(M)$ The matrix $M(a,b,R)$ is invertible.\\

Now we state our main result for the estimator $\widehat\alpha^{(W)}$.
Throughout, $\overset{D}{\longrightarrow}$ denotes convergence in distribution, and
limiting and order relations are always meant as $n\to\infty$ if not specified otherwise.
\begin{thm}\label{mt}
	Assume that the conditions $Q_{1}-Q_{3}$ are satisfied for the underlying distribution and suppose that the quantile function $Q$ admits the representation (\ref{Qrepr}).  Moreover, assume the conditions $(R)$ and $(M)$, and assume also that the percentiles $s_j$ are chosen from a closed set $U=[a,b]$, $0<a<b<1$, such that $s_j=j/n$, $j=\lceil na\rceil,\ldots,\lfloor nb\rfloor$, and $\widetilde p>p$. Then	
		
	\begin{equation}\label{asnormal}
	\sqrt n(\widehat\alpha_n^{(W)}-\alpha)\overset{D}{\longrightarrow}N(0,V),
	\end{equation} 
		where
	\begin{equation}\label{var}
	V=\int_a^b\int_a^b\frac{G_{R}(s)G_{R}(t)\big((1-s)\land (1-t)-(1-s)(1-t)\big)}{Q(1-s)Q(1-t)fQ(1-s)fQ(1-t)}dsdt.
	\end{equation}
\end{thm}

\bigskip\noindent

The proof is in Chapter 3.	

\section{Simulation results}
In order to make a comparison with existing proposals, simulations were done performed by the Matlab software. The  samples were generated from the strict Pareto model $L\equiv1$ in (\ref{quantile}) and from the Hall model (\ref{Hallm}). The Hill, Pickands, DEdH (Dekkers, Einmahl and de Haan) and the weighted least squares (WLS) estimators were included in the simulation study. We used the values $n=5000$, $a=0.001$, $b=0.4$ and $\widetilde p=1,2,3$, and the weight function $R(s)=s/500$ for the WLS estimator. In case of $R\equiv1$, we refer to as ordinary least squares (OLS) estimator. The tail indexes were chosen between $0.5$ and $20$. For the Hill, Pickands and DEdH estimators the simulations were done for sample size $n=5000$ and sample fraction size $k_n=200$. All the simulations were repeated 1000 times.

Tables \ref{table_mse1} and \ref{table_mean1} contains the empirical mean square errors (MSE) and the average simulated estimates (mean) for the strict Pareto model. We conclude that in the submodel $L\equiv1$ for all $\alpha$ values, the WLS estimator performs better than the other estimators investigated.

Tables \ref{table_mse2} and \ref{table_mean2} presents the simulation results for the Hall model.
Specifically, we used the parameters $D_1=0.4$, $D_2=1$ and $\beta=0.01$. We see from Table \ref{table_mse2} that the WLS estimator performs better than the other estimators, and the OLS estimator is competitive with the Hill estimator especially for $\widetilde p=3$. 

\begin{table}[!h]
	\caption{ Empirical mean square errors (MSE) of tail index estimates for the Pareto model and for sample size $n=5000$.}
	\label{table_mse1}
	\scalebox{0.70}{
		\begin{tabular}{cccclllccc}
			\hline
			\multicolumn{10}{c}{MSE}                                                                                                                                                                                                                                     \\ \cline{2-10} 
			& \multicolumn{3}{c|}{WLS}                            & \multicolumn{3}{c|}{OLS}                            & \multicolumn{1}{c|}{\multirow{2}{*}{Hill}} & \multicolumn{1}{c|}{\multirow{2}{*}{Pickands}} & \multirow{2}{*}{DEdh} \\ \cline{1-7}
			\multicolumn{1}{c|}{$\alpha$} & $\widetilde p=1$      & $\widetilde p=2$        & \multicolumn{1}{c|}{$\widetilde p=3$}   & $\widetilde p=1$        & $\widetilde p=2$        &  \multicolumn{1}{c|}{$\widetilde p=3$} & \multicolumn{1}{c|}{}                      & \multicolumn{1}{c|}{}                          &                       \\ \hline
			\multicolumn{1}{c|}{0.5}   & 0.00049  & 0.000668 & \multicolumn{1}{c|}{0.000945} & 0.00065  & 0.00098  & \multicolumn{1}{l|}{0.001357} & \multicolumn{1}{c|}{0.001172}              & \multicolumn{1}{c|}{0.017866}                  & 0.006558              \\
			\multicolumn{1}{c|}{0.8}   & 0.001183 & 0.001572 & \multicolumn{1}{c|}{0.002261} & 0.00161  & 0.002368 & \multicolumn{1}{l|}{0.00325}  & \multicolumn{1}{c|}{0.003325}              & \multicolumn{1}{c|}{0.02146}                   & 0.008336              \\
			\multicolumn{1}{c|}{1}     & 0.001756 & 0.002394 & \multicolumn{1}{c|}{0.003668} & 0.002425 & 0.003697 & \multicolumn{1}{l|}{0.005203} & \multicolumn{1}{c|}{0.005457}              & \multicolumn{1}{c|}{0.024083}                  & 0.010687              \\
			\multicolumn{1}{c|}{1.2}   & 0.002821 & 0.003826 & \multicolumn{1}{c|}{0.005298} & 0.003641 & 0.005365 & \multicolumn{1}{l|}{0.007366} & \multicolumn{1}{c|}{0.007532}              & \multicolumn{1}{c|}{0.025102}                  & 0.01219               \\
			\multicolumn{1}{c|}{1.5}   & 0.00451  & 0.006126 & \multicolumn{1}{c|}{0.008397} & 0.005867 & 0.008671 & \multicolumn{1}{l|}{0.01188}  & \multicolumn{1}{c|}{0.01052}               & \multicolumn{1}{c|}{0.03013}                   & 0.016092              \\
			\multicolumn{1}{c|}{1.8}   & 0.006049 & 0.007993 & \multicolumn{1}{c|}{0.011399} & 0.007694 & 0.011178 & \multicolumn{1}{l|}{0.015334} & \multicolumn{1}{c|}{0.016801}              & \multicolumn{1}{c|}{0.035497}                  & 0.021695              \\
			\multicolumn{1}{c|}{2}     & 0.007639 & 0.010499 & \multicolumn{1}{c|}{0.014921} & 0.010842 & 0.016055 & \multicolumn{1}{l|}{0.022093} & \multicolumn{1}{c|}{0.020194}              & \multicolumn{1}{c|}{0.034981}                  & 0.025421              \\
			\multicolumn{1}{c|}{3}     & 0.017668 & 0.024202 & \multicolumn{1}{c|}{0.034858} & 0.023523 & 0.034985 & \multicolumn{1}{l|}{0.047931} & \multicolumn{1}{c|}{0.044665}              & \multicolumn{1}{c|}{0.063986}                  & 0.049712              \\
			\multicolumn{1}{c|}{4}     & 0.029136 & 0.040729 & \multicolumn{1}{c|}{0.05895}  & 0.03926  & 0.058641 & \multicolumn{1}{l|}{0.080589} & \multicolumn{1}{c|}{0.0807}                & \multicolumn{1}{c|}{0.094346}                  & 0.089062              \\
			\multicolumn{1}{c|}{5}     & 0.047688 & 0.063472 & \multicolumn{1}{c|}{0.096547} & 0.064079 & 0.094958 & \multicolumn{1}{l|}{0.13097}  & \multicolumn{1}{c|}{0.114725}              & \multicolumn{1}{c|}{0.13557}                   & 0.121162              \\
			\multicolumn{1}{c|}{5.5}   & 0.055014 & 0.076889 & \multicolumn{1}{c|}{0.106532} & 0.074036 & 0.110494 & \multicolumn{1}{l|}{0.151476} & \multicolumn{1}{c|}{0.142506}              & \multicolumn{1}{c|}{0.16283}                   & 0.144236              \\
			\multicolumn{1}{c|}{6}     & 0.071694 & 0.103854 & \multicolumn{1}{c|}{0.141469} & 0.089924 & 0.129628 & \multicolumn{1}{l|}{0.171023} & \multicolumn{1}{c|}{0.173129}              & \multicolumn{1}{c|}{0.188113}                  & 0.175776              \\
			\multicolumn{1}{c|}{10}    & 0.191172 & 0.262768 & \multicolumn{1}{c|}{0.375258} & 0.233466 & 0.339353 & \multicolumn{1}{l|}{0.45505}  & \multicolumn{1}{c|}{0.525182}              & \multicolumn{1}{c|}{0.558138}                  & 0.527627              \\
			\multicolumn{1}{c|}{15}    & 0.402501 & 0.535825 & \multicolumn{1}{c|}{0.802723} & 0.582015 & 0.884501 & \multicolumn{1}{l|}{1.226799} & \multicolumn{1}{c|}{1.169978}              & \multicolumn{1}{c|}{1.167519}                  & 1.176961              \\
			\multicolumn{1}{c|}{20}    & 0.792631 & 1.095608 & \multicolumn{1}{c|}{1.579634} & 0.996911 & 1.434474 & \multicolumn{1}{l|}{1.916717} & \multicolumn{1}{c|}{2.100758}              & \multicolumn{1}{c|}{1.981171}                  & 2.101663              \\ \hline\\
	\end{tabular}}

\end{table}

\vfill\eject
\begin{table}[!h]
	\caption{Average simulated tail index estimates (Mean) for sample size $n=5000$ and for the Pareto model.}
	\label{table_mean1}
	\scalebox{0.72}{
		\begin{tabular}{cccclllccc}	
			\hline
			\multicolumn{10}{c}{Mean}                                                                                                                                                                                                                                 \\ \cline{2-10} 
			& \multicolumn{3}{c|}{WLS}                            & \multicolumn{3}{c|}{OLS}                            & \multicolumn{1}{c|}{\multirow{2}{*}{Hill}} & \multicolumn{1}{c|}{\multirow{2}{*}{Pickands}} & \multirow{2}{*}{DEdh} \\ \cline{1-7}
			\multicolumn{1}{c|}{$\alpha$}& $\widetilde p=1$      & $\widetilde p=2$        & \multicolumn{1}{c|}{$\widetilde p=3$}   & $\widetilde p=1$        & $\widetilde p=2$        &  \multicolumn{1}{c|}{$\widetilde p=3$}        & \multicolumn{1}{c|}{}                      & \multicolumn{1}{c|}{}                          &                       \\ \hline
			\multicolumn{1}{c|}{0.5}   & 0.500964 & 0.501233 & \multicolumn{1}{c|}{0.502571} & 0.503044 & 0.504023 & \multicolumn{1}{l|}{0.505077} & \multicolumn{1}{c|}{0.501476}              & \multicolumn{1}{c|}{0.495427}                  & 0.489674              \\
			\multicolumn{1}{c|}{0.8}   & 0.801937 & 0.802524 & \multicolumn{1}{c|}{0.803656} & 0.805577 & 0.807293 & \multicolumn{1}{l|}{0.809021} & \multicolumn{1}{c|}{0.800238}              & \multicolumn{1}{c|}{0.801774}                  & 0.783686              \\
			\multicolumn{1}{c|}{1}     & 1.001483 & 1.001634 & \multicolumn{1}{c|}{1.00246}  & 1.005316 & 1.00711  & \multicolumn{1}{l|}{1.009101} & \multicolumn{1}{c|}{1.001825}              & \multicolumn{1}{c|}{1.004785}                  & 0.98694               \\
			\multicolumn{1}{c|}{1.2}   & 1.201603 & 1.201804 & \multicolumn{1}{c|}{1.202563} & 1.206612 & 1.208947 & \multicolumn{1}{l|}{1.211492} & \multicolumn{1}{c|}{1.197918}              & \multicolumn{1}{c|}{1.195252}                  & 1.185589              \\
			\multicolumn{1}{c|}{1.5}   & 1.502324 & 1.502346 & \multicolumn{1}{c|}{1.502635} & 1.509168 & 1.512328 & \multicolumn{1}{l|}{1.515847} & \multicolumn{1}{c|}{1.501775}              & \multicolumn{1}{c|}{1.492907}                  & 1.485452              \\
			\multicolumn{1}{c|}{1.8}   & 1.805614 & 1.807831 & \multicolumn{1}{c|}{1.808328} & 1.812501 & 1.815819 & \multicolumn{1}{l|}{1.818663} & \multicolumn{1}{c|}{1.801355}              & \multicolumn{1}{c|}{1.80158}                   & 1.787262              \\
			\multicolumn{1}{c|}{2}     & 2.006075 & 2.008649 & \multicolumn{1}{c|}{2.012745} & 2.016946 & 2.022076 & \multicolumn{1}{l|}{2.026978} & \multicolumn{1}{c|}{2.004505}              & \multicolumn{1}{c|}{2.004395}                  & 1.988554              \\
			\multicolumn{1}{c|}{3}     & 3.004755 & 3.002857 & \multicolumn{1}{c|}{3.007692} & 3.013462 & 3.017458 & \multicolumn{1}{l|}{3.022898} & \multicolumn{1}{c|}{3.007171}              & \multicolumn{1}{c|}{3.002503}                  & 2.996076              \\
			\multicolumn{1}{c|}{4}     & 4.00635  & 4.009942 & \multicolumn{1}{c|}{4.017468} & 4.028563 & 4.039037 & \multicolumn{1}{l|}{4.049668} & \multicolumn{1}{c|}{3.985504}              & \multicolumn{1}{c|}{3.98685}                   & 3.966318              \\
			\multicolumn{1}{c|}{5}     & 5.007934 & 5.007172 & \multicolumn{1}{c|}{5.011766} & 5.020999 & 5.027234 & \multicolumn{1}{l|}{5.034629} & \multicolumn{1}{c|}{5.004943}              & \multicolumn{1}{c|}{5.012502}                  & 4.98503               \\
			\multicolumn{1}{c|}{5.5}   & 5.521636 & 5.523414 & \multicolumn{1}{c|}{5.535038} & 5.54912  & 5.562017 & \multicolumn{1}{l|}{5.576119} & \multicolumn{1}{c|}{5.498843}              & \multicolumn{1}{c|}{5.49632}                   & 5.48765               \\
			\multicolumn{1}{c|}{6}     & 6.010705 & 6.020936 & \multicolumn{1}{c|}{6.035309} & 6.042542 & 6.057651 & \multicolumn{1}{l|}{6.071267} & \multicolumn{1}{c|}{6.00263}               & \multicolumn{1}{c|}{6.012857}                  & 5.987134              \\
			\multicolumn{1}{c|}{10}    & 10.03551 & 10.0453  & \multicolumn{1}{c|}{10.04212} & 10.06879 & 10.0851  & \multicolumn{1}{l|}{10.099}   & \multicolumn{1}{c|}{9.997173}              & \multicolumn{1}{c|}{10.04161}                  & 9.981231              \\
			\multicolumn{1}{c|}{15}    & 15.00041 & 15.02029 & \multicolumn{1}{c|}{15.05347} & 15.07633 & 15.11221 & \multicolumn{1}{l|}{15.14596} & \multicolumn{1}{c|}{15.05984}              & \multicolumn{1}{c|}{15.02914}                  & 15.0449               \\
			\multicolumn{1}{c|}{20}    & 20.0481  & 20.05749 & \multicolumn{1}{c|}{20.09294} & 20.11033 & 20.14008 & \multicolumn{1}{l|}{20.17114} & \multicolumn{1}{c|}{20.01204}              & \multicolumn{1}{c|}{20.04928}                  & 19.99807              \\ \hline
	\end{tabular}}
\end{table}

\begin{table}[!h]
	\caption{ Empirical mean square errors (MSE) of tail index estimates for the Hall model and for sample size $n=5000$.}
	\label{table_mse2}
	\scalebox{0.69}{
		\begin{tabular}{llllllllll}
			
			\hline
			\multicolumn{10}{c}{MSE}                                                                                                                                                                                                                                                                                                  \\ \cline{2-10} 
			& \multicolumn{3}{c|}{WLS}                                                                      & \multicolumn{3}{c|}{OLS}                            & \multicolumn{1}{c|}{\multirow{2}{*}{Hill}} & \multicolumn{1}{c|}{\multirow{2}{*}{Pickands}} & \multicolumn{1}{c}{\multirow{2}{*}{DEdh}} \\ \cline{1-7}
			\multicolumn{1}{l|}{$\alpha$} & \multicolumn{1}{c}{$\tilde{p}=1$}      & \multicolumn{1}{c}{$\tilde{p}=2$}        & \multicolumn{1}{c|}{$\tilde{p}=3$}          & $\tilde{p}=1$        & $\tilde{p}=2$        & \multicolumn{1}{l|}{$\tilde{p}=3$}        & \multicolumn{1}{c|}{}                      & \multicolumn{1}{c|}{}                          & \multicolumn{1}{c}{}                      \\ \hline
			\multicolumn{1}{l|}{0.5}  & \multicolumn{1}{c}{0.000495} & \multicolumn{1}{c}{0.000667} & \multicolumn{1}{c|}{0.00092558} & 0.000632 & 0.000946 & \multicolumn{1}{l|}{0.001306} & \multicolumn{1}{c|}{0.001159}              & \multicolumn{1}{l|}{0.017902}                  & \multicolumn{1}{c}{0.00665892}            \\
			\multicolumn{1}{l|}{0.8}  & \multicolumn{1}{c}{0.001174} & \multicolumn{1}{c}{0.001552} & \multicolumn{1}{c|}{0.00222172} & 0.00156  & 0.002292 & \multicolumn{1}{l|}{0.003147} & \multicolumn{1}{c|}{0.003306}              & \multicolumn{1}{l|}{0.02142}                   & \multicolumn{1}{c}{0.00847904}            \\
			\multicolumn{1}{l|}{1}    & \multicolumn{1}{c}{0.001749} & \multicolumn{1}{c}{0.002379} & \multicolumn{1}{c|}{0.00363231} & 0.002374 & 0.003616 & \multicolumn{1}{l|}{0.005088} & \multicolumn{1}{c|}{0.00541}               & \multicolumn{1}{l|}{0.024003}                  & \multicolumn{1}{c}{0.01078627}            \\
			\multicolumn{1}{l|}{1.2}  & \multicolumn{1}{c}{0.002806} & \multicolumn{1}{c}{0.003801} & \multicolumn{1}{c|}{0.00525345} & 0.003571 & 0.005259 & \multicolumn{1}{l|}{0.007218} & \multicolumn{1}{c|}{0.007516}              & \multicolumn{1}{l|}{0.025114}                  & \multicolumn{1}{c}{0.01229618}            \\
			\multicolumn{1}{l|}{1.5}  & \multicolumn{1}{c}{0.004482} & \multicolumn{1}{c}{0.006087} & \multicolumn{1}{c|}{0.00834029} & 0.005763 & 0.008519 & \multicolumn{1}{l|}{0.011673} & \multicolumn{1}{c|}{0.010459}              & \multicolumn{1}{l|}{0.030153}                  & \multicolumn{1}{c}{0.01618835}            \\
			\multicolumn{1}{l|}{1.8}  & \multicolumn{1}{c}{0.005985} & \multicolumn{1}{c}{0.007897} & \multicolumn{1}{c|}{0.01127938} & 0.007554 & 0.010987 & \multicolumn{1}{l|}{0.015093} & \multicolumn{1}{c|}{0.016721}              & \multicolumn{1}{l|}{0.035417}                  & \multicolumn{1}{c}{0.02175322}            \\
			\multicolumn{1}{l|}{2}    & 0.007566                     & 0.010387                     & \multicolumn{1}{l|}{0.01474723} & 0.010648 & 0.015785 & \multicolumn{1}{l|}{0.021747} & \multicolumn{1}{l|}{0.020076}              & \multicolumn{1}{l|}{0.034877}                  & 0.02545883                                \\
			\multicolumn{1}{l|}{3}    & 0.017587                     & 0.024119                     & \multicolumn{1}{l|}{0.03469301} & 0.023338 & 0.034725 & \multicolumn{1}{l|}{0.047576} & \multicolumn{1}{l|}{0.044474}              & \multicolumn{1}{l|}{0.063841}                  & 0.04963012                                \\
			\multicolumn{1}{l|}{4}    & 0.029026                     & 0.040556                     & \multicolumn{1}{l|}{0.0586581}  & 0.038909 & 0.058141 & \multicolumn{1}{l|}{0.079932} & \multicolumn{1}{l|}{0.08067}               & \multicolumn{1}{l|}{0.094312}                  & 0.08921482                                \\
			\multicolumn{1}{l|}{5}    & 0.04754                      & 0.063301                     & \multicolumn{1}{l|}{0.09626703} & 0.063773 & 0.094531 & \multicolumn{1}{l|}{0.130401} & \multicolumn{1}{l|}{0.114477}              & \multicolumn{1}{l|}{0.135233}                  & 0.12110866                                \\
			\multicolumn{1}{l|}{5.5}  & 0.054727                     & 0.076546                     & \multicolumn{1}{l|}{0.10602299} & 0.073448 & 0.109716 & \multicolumn{1}{l|}{0.150488} & \multicolumn{1}{l|}{0.142289}              & \multicolumn{1}{l|}{0.162625}                  & 0.14413155                                \\
			\multicolumn{1}{l|}{6}    & 0.071496                     & 0.103502                     & \multicolumn{1}{l|}{0.14091586} & 0.089385 & 0.128878 & \multicolumn{1}{l|}{0.170073} & \multicolumn{1}{l|}{0.172846}              & \multicolumn{1}{l|}{0.187722}                  & 0.17564752                                \\
			\multicolumn{1}{l|}{10}   & 0.190659                     & 0.262089                     & \multicolumn{1}{l|}{0.37450066} & 0.232588 & 0.338214 & \multicolumn{1}{l|}{0.453664} & \multicolumn{1}{l|}{0.524723}              & \multicolumn{1}{l|}{0.557207}                  & 0.52732507                                \\
			\multicolumn{1}{l|}{15}   & 0.402258                     & 0.5353                       & \multicolumn{1}{l|}{0.80169824} & 0.580913 & 0.882852 & \multicolumn{1}{l|}{1.2246}   & \multicolumn{1}{l|}{1.168656}              & \multicolumn{1}{l|}{1.166491}                  & 1.17578666                                \\
			\multicolumn{1}{l|}{20}   & 0.791792                     & 1.094529                     & \multicolumn{1}{l|}{1.57797168} & 0.995368 & 1.432428 & \multicolumn{1}{l|}{1.914136} & \multicolumn{1}{l|}{2.099641}              & \multicolumn{1}{l|}{1.979735}                  & 2.10068457                                \\ \hline
	\end{tabular}}
\end{table}

\begin{table}[!h]
	\caption{Average simulated tail index estimates (Mean) for sample size $n=5000$ and for the Hall model.}
	\label{table_mean2}
	\scalebox{0.75}{
		\begin{tabular}{llllllllll}
			\hline
			\multicolumn{10}{c}{Mean}                                                                                                                                                                                                                                                                                               \\ \cline{2-10} 
			& \multicolumn{3}{c|}{WLS}                                                                    & \multicolumn{3}{c|}{OLS}                            & \multicolumn{1}{c|}{\multirow{2}{*}{Hill}} & \multicolumn{1}{c|}{\multirow{2}{*}{Pickands}} & \multicolumn{1}{c}{\multirow{2}{*}{DEdh}} \\ \cline{1-7}
			\multicolumn{1}{l|}{$\alpha$} & \multicolumn{1}{c}{$\tilde{p}=1$}      & \multicolumn{1}{c}{$\tilde{p}=2$}        & \multicolumn{1}{c|}{$\tilde{p}=3$}        & $\tilde{p}=1$        & $\tilde{p}=2$        & \multicolumn{1}{l|}{$\tilde{p}=3$}        & \multicolumn{1}{c|}{}                      & \multicolumn{1}{c|}{}                          & \multicolumn{1}{c}{}                      \\ \hline
			\multicolumn{1}{l|}{0.5}  & \multicolumn{1}{c}{0.49603}  & \multicolumn{1}{c}{0.496302} & \multicolumn{1}{c|}{0.497636} & 0.498107 & 0.499084 & \multicolumn{1}{l|}{0.500135} & \multicolumn{1}{c|}{0.496567}              & \multicolumn{1}{l|}{0.490542}                  & \multicolumn{1}{c}{0.484814}              \\
			\multicolumn{1}{l|}{0.8}  & \multicolumn{1}{c}{0.797}    & \multicolumn{1}{c}{0.79759}  & \multicolumn{1}{c|}{0.798724} & 0.800636 & 0.802349 & \multicolumn{1}{l|}{0.804074} & \multicolumn{1}{c|}{0.795342}              & \multicolumn{1}{l|}{0.796859}                  & \multicolumn{1}{c}{0.77882}               \\
			\multicolumn{1}{l|}{1}    & \multicolumn{1}{c}{0.996551} & \multicolumn{1}{c}{0.996707} & \multicolumn{1}{c|}{0.997539} & 1.000382 & 1.002176 & \multicolumn{1}{l|}{1.004164} & \multicolumn{1}{c|}{0.996921}              & \multicolumn{1}{l|}{0.999856}                  & \multicolumn{1}{c}{0.982061}              \\
			\multicolumn{1}{l|}{1.2}  & \multicolumn{1}{c}{1.196672} & \multicolumn{1}{c}{1.196878} & \multicolumn{1}{c|}{1.197643} & 1.201678 & 1.204011 & \multicolumn{1}{l|}{1.206553} & \multicolumn{1}{c|}{1.193032}              & \multicolumn{1}{l|}{1.190336}                  & \multicolumn{1}{c}{1.180723}              \\
			\multicolumn{1}{l|}{1.5}  & \multicolumn{1}{c}{1.497391} & \multicolumn{1}{c}{1.49742}  & \multicolumn{1}{c|}{1.497717} & 1.50423  & 1.507388 & \multicolumn{1}{l|}{1.510903} & \multicolumn{1}{c|}{1.496874}              & \multicolumn{1}{l|}{1.487989}                  & \multicolumn{1}{c}{1.480568}              \\
			\multicolumn{1}{l|}{1.8}  & \multicolumn{1}{c}{1.800674} & \multicolumn{1}{c}{1.802891} & \multicolumn{1}{c|}{1.803397} & 1.807559 & 1.810876 & \multicolumn{1}{l|}{1.81372}  & \multicolumn{1}{c|}{1.796457}              & \multicolumn{1}{l|}{1.796655}                  & \multicolumn{1}{c}{1.782377}              \\
			\multicolumn{1}{l|}{2}    & 2.001136                     & 2.003709                     & \multicolumn{1}{l|}{2.007804} & 2.011997 & 2.017123 & \multicolumn{1}{l|}{2.02202}  & \multicolumn{1}{l|}{1.999599}              & \multicolumn{1}{l|}{1.999456}                  & 1.98366                                   \\
			\multicolumn{1}{l|}{3}    & 2.999823                     & 2.997934                     & \multicolumn{1}{l|}{3.00277}  & 3.008533 & 3.01253  & \multicolumn{1}{l|}{3.017969} & \multicolumn{1}{l|}{3.002265}              & \multicolumn{1}{l|}{2.99757}                   & 2.991178                                  \\
			\multicolumn{1}{l|}{4}    & 4.001418                     & 4.005012                     & \multicolumn{1}{l|}{4.012537} & 4.023621 & 4.03409  & \multicolumn{1}{l|}{4.044716} & \multicolumn{1}{l|}{3.980627}              & \multicolumn{1}{l|}{3.981932}                  & 3.961447                                  \\
			\multicolumn{1}{l|}{5}    & 5.003001                     & 5.002247                     & \multicolumn{1}{l|}{5.006845} & 5.016071 & 5.022308 & \multicolumn{1}{l|}{5.029703} & \multicolumn{1}{l|}{5.000043}              & \multicolumn{1}{l|}{5.007562}                  & 4.980135                                  \\
			\multicolumn{1}{l|}{5.5}  & 5.516692                     & 5.518475                     & \multicolumn{1}{l|}{5.530098} & 5.544169 & 5.557062 & \multicolumn{1}{l|}{5.57116}  & \multicolumn{1}{l|}{5.493949}              & \multicolumn{1}{l|}{5.491392}                  & 5.482761                                  \\
			\multicolumn{1}{l|}{6}    & 6.005772                     & 6.016001                     & \multicolumn{1}{l|}{6.03037}  & 6.037599 & 6.052704 & \multicolumn{1}{l|}{6.066316} & \multicolumn{1}{l|}{5.997733}              & \multicolumn{1}{l|}{6.007918}                  & 5.982241                                  \\
			\multicolumn{1}{l|}{10}   & 10.03057                     & 10.04036                     & \multicolumn{1}{l|}{10.03719} & 10.06385 & 10.08015 & \multicolumn{1}{l|}{10.09406} & \multicolumn{1}{l|}{9.99228}               & \multicolumn{1}{l|}{10.03666}                  & 9.97634                                   \\
			\multicolumn{1}{l|}{15}   & 14.99548                     & 15.01536                     & \multicolumn{1}{l|}{15.04854} & 15.07139 & 15.10728 & \multicolumn{1}{l|}{15.14102} & \multicolumn{1}{l|}{15.05493}              & \multicolumn{1}{l|}{15.0242}                   & 15.03999                                  \\
			\multicolumn{1}{l|}{20}   & 20.04316                     & 20.05255                     & \multicolumn{1}{l|}{20.08801} & 20.1054  & 20.13515 & \multicolumn{1}{l|}{20.16621} & \multicolumn{1}{l|}{20.00714}              & \multicolumn{1}{l|}{20.04434}                  & 19.99317                                  \\ \hline
	\end{tabular}}
\end{table}

\newpage

\section{Proof of Theorem \ref{mt}}

Let $q_n(s)$ be the quantile process defined as
\begin{equation*}
q_n(s)=\sqrt n(Q_n(s)-Q(s)),\quad 0<s<1.
\end{equation*}

The proof is based on the strong approximation of the quantile process. 
\begin{thm}\label{qproct}
	(\cite{Csorgo1978}, Theorem 6.)
    Suppose that the conditions $Q_{1}$ and $Q_{2}$ are satisfied. Then on some probability space one can define a sequence $\{B_n(t):0\le t\le1\}_{n=1}^\infty$ of Brownian bridges such that
	\begin{equation*}
	\sup_{\delta_n\le s\le1-\delta_n}\big|fQ(s)q_n(s)-B_n(s)\big|\overset{a.s.}=O(n^{-1/2}\log n),
	\end{equation*}
	where $\delta_n=25n^{-1}\log\log n$.
\end{thm}
\begin{proof}[\bf Proof of Theorem \ref{mt}.]
	We assume that the random variables $X_1,X_2,\ldots$ are defined on the probability space given in Theorem \ref{qproct}.
	By a simple calculation,
	\begin{small}
		\[X'WX=\begin{bmatrix}
		\smashoperator\sum_{j=\lceil na\rceil}^{\lfloor nb\rfloor}\log^2s_jR(s_j)
		&-\smashoperator\sum_{j=\lceil na\rceil}^{\lfloor nb\rfloor}\log s_jR(s_j) & -2\smashoperator\sum_{j=\lceil na\rceil}^{\lfloor nb\rfloor}\log s_j\cos(2\pi s_j)R(s_j)\ldots\\
	-\smashoperator \sum_{j=\lceil na\rceil}^{\lfloor nb\rfloor}\log s_jR(s_j) & \sum_{j=\lceil na\rceil}^{\lfloor nb\rfloor}R(s_j) & 2\smashoperator\sum_{ j=\lceil na\rceil}^{\lfloor nb\rfloor}\cos(2\pi s_j)R(s_j)\ldots\\
		\vdots & \vdots & \vdots \end{bmatrix}.\]
	\end{small}	
By Riemann sum approximation, we get
	\begin{align}\label{mlimit}
	\lim_{n\to\infty}&n^{-1}X'WX=M(a,b,R)\\&:=
	\begin{bmatrix}
	\int_a^b\log^2u\,R(u)du
	& -\int_a^b\log u\,R(u)du & -2\int_a^b\log u\cos(2\pi u)R(u)du\ldots\\
	-\int_a^b\log u\,R(u)du & \int_a^bR(u)du & 2\int_a^b\cos(2\pi u)R(u)du\ldots\\
	\vdots & \vdots & \vdots \end{bmatrix}.\nonumber
	\end{align}
	Set $\underline\varepsilon:=\big(\varepsilon(s_{\lceil na\rceil}),\ldots,\varepsilon(s_{\lfloor nb\rfloor})\big)'$ and
	\[y^*:=\big(\log Q(1-s_{\lceil na\rceil}),\ldots,\log Q(1-s_{\lfloor nb\rfloor}))\big).\]
	Then we have $y^*=X\beta_{\widetilde p}$, $\beta_{\widetilde p}=(X'WX)^{-1}X'Wy^*$ and hence $\alpha=e_1'\beta_{\widetilde p}=e_1'(X'WX)^{-1}X'Wy^*$.
	It follows that $\underline\varepsilon=y-y^*$ and
	\[\sqrt n(\widehat\alpha_n^{(W)}-\alpha)=\frac1{\sqrt n}e'_1(n^{-1}X'WX)^{-1}X'W\underline\varepsilon=Y_n+A_n,\]
	where $Y_n=n^{-1/2}e'_{1}M(a,b,R)^{-1}X'W\underline\varepsilon$ and
	\[A_n=n^{-1/2}e'_1\big((n^{-1}X'WX)^{-1}-M(a,b,R)^{-1}\big)X'W\underline\varepsilon.\]
	A straightforward calculation yields
	\begin{equation}\label{Yn}
	Y_n=\frac1{\sqrt n}\sum_{j=\lceil na\rceil}^{ \lfloor nb\rfloor}\varepsilon(s_j)G_{R}(s_j).
	\end{equation}
	The main point of the proof is to show that 
	\begin{equation}\label{Yconv}
	\frac1{\sqrt n}\sum_{j=\lceil na\rceil}^{ \lfloor nb\rfloor}\varepsilon(s_j)G_{R}(s_j)\overset{D}{\longrightarrow}N(0,V).
	\end{equation}
	With $\gamma_n(s):=(Q_n(1-s)-Q(1-s))/{Q(1-s)}$, the residual process can be written as $\varepsilon(s)=\log(1+\gamma_n(s))$.
	Set $\eta(x):=\log(1+x)-x$, and let $C$ and $\delta$ be some constants such that $\eta(x)\le Cx^2$, if $|x|\le\delta$.
	Then we obtain $Y_n=Y_{n,1}+A_{n,1}$, where
	\[Y_{n,1}=\frac1{\sqrt n}\sum_{j=\lceil na\rceil}^{ \lfloor nb\rfloor}\gamma_n(s_j)G_{R}(s_j),\quad
	A_{n,1}=\frac1{\sqrt n}\sum_{j=\lceil na\rceil}^{ \lfloor nb\rfloor}\eta(\gamma_n(s_j))G_{R}(s_j).\]
	First we show that $A_{n,1}=o_P(1)$. On the event
	\[
	E_n:=\left\{\max_{\lceil na\rceil \le j\le \lfloor nb\rfloor}|\gamma_n(s_j)|\le\delta\right\},
	\]
	we have 
	\[
	|A_{n,1}|\le C\sqrt n\max_{\lceil na\rceil \le j\le \lfloor nb\rfloor}\gamma^2_n(s_j)
	\,\frac1n\sum_{j=\lceil na\rceil}^{ \lfloor nb\rfloor}|G_{R}(s_j)|.
	\]
	With $\kappa_1:=\sup_{1-b\le s\le1-a}1/|Q(s)|$, we obtain
	\[
	\max_{\lceil na\rceil \le j\le \lfloor nb\rfloor}\gamma^2_n(s_j)
	\le\kappa_1^2\sup_{1-b\le s\le1-a}(Q_n(s)-Q(s))^2.
	\]
	Set $e_n(s):=fQ(s)q_n(s)-B_n(s)$. With the Brownian bridges in Theorem \ref{qproct} and $\kappa_2:=\sup_{1-b\le s\le1-a}1/fQ(s)$ we get
	\begin{align*}
	\sup_{1-b\le s\le1-a}|Q_n(s)-Q(s)|
	&=\frac1{\sqrt n}\sup_{1-b\le s\le1-a}\frac{|e_n(s)+B_n(s)|}{fQ(s)}\cr
	&\le\frac{\kappa_2}{\sqrt n}\sup_{1-b\le s\le1-a}(|e_n(s)|+|B_n(s)|).
	\end{align*}
	It follows that
	\[\sqrt n\max_{\lceil na\rceil \le j\le \lfloor nb\rfloor}\gamma^2_n(s_j)
	\le\frac{\kappa_1^2\kappa_2^2}{\sqrt n}\left(\sup_{1-b\le s\le1-a}|e_n(s)|+\sup_{1-b\le s\le1-a}|B_n(s)|\right)^2.\]
	
	Applying Theorem \ref{qproct}, we obtain $\sqrt n\max_{\lceil na\rceil \le j\le \lfloor nb\rfloor}\gamma^2_n(s_j)=o_P(1)$. 
	This, in combination with $P(E_n)\to0$ and $\frac1n\sum_{j=\lceil na\rceil}^{ \lfloor nb\rfloor}|G_{R}(s_j)|\to\int_a^b|G_R(s)ds$ implies $A_{n,1}=o_P(1)$.
	
	Now we decompose $Y_{n,1}$ as $Y_{n,1}=Y_{n,2}+A_{n,2}$, where
	\begin{align*}Y_{n,2}&=\frac1n\sum_{j=\lceil na\rceil}^{\lfloor nb\rfloor}\frac{B_n(1-s_j)G_R(s_j)}{fQ(1-s_j)Q(1-s_j)},\cr
	A_{n,2}&=\frac1n\sum_{j=\lceil na\rceil}^{\lfloor nb\rfloor}\frac{e_n(1-s_j)}{fQ(1-s_j)Q(1-s_j)}G_R(s_j).
	\end{align*}
	To prove that $A_{n,2}=o_P(1)$, we use the inequality
	\[A_{n,2}\le\kappa_3\sup_{1-b\le s\le1-a}|e_n(s)|\,\frac1n\sum_{j=\lceil na\rceil}^{\lfloor nb\rfloor}|G_R(s_j)|,\]
	where
	\[\kappa_3=\sup_{1-b\le s\le 1-a}1/|fQ(s)Q(s)|.\]
	By Theorem \ref{qproct} we have $A_{n,2}=o_P(1)$. We prove that the limit of $Y_{n,2}$ is $N(0,V)$ given in (\ref{asnormal}). By the distributional equality
	\[Y_{n,2}\overset{D}=\frac1n\sum_{j=\lceil na\rceil}^{\lfloor nb\rfloor}\frac{B(1-s_j)G_R(s_j)}{fQ(1-s_j)Q(1-s_j)},\quad n=1,2,\ldots,\]
	where $B(\cdot)$ is a Brownian bridge process, we obtain
	\[Y_{n,2}\overset{D}{\longrightarrow}\int_a^b\frac{B(1-s)G_R(s)}{fQ(1-s)Q(1-s)}ds.\]
	The variance of the limit random variable is described in (\ref{var}).
	
	The last step is to prove that $A_n=o_P(1)$. Let $(v^*_n, v_{0,n},\ldots, v_{\widetilde p,n})$ be the first row of $(n^{-1}X'WX)^{-1}-M(a,b,R)^{-1}$.
	Using statement (\ref{mlimit}), we have $(v^*_n, v_{0,n},\ldots, v_{\widetilde p,n})\to\,\mathbf0$.
	Set
	\[G^{(n)}(u):=R(u)\big(-v^*_n\log u+v_{0,n}+2\sum_{k=1}^{\widetilde p}v_{k,n}\cos(2\pi ku)\big),\quad u\in(0,1).\]
	Similarly as in (\ref{Yn}),
	\begin{align*}
	A_n&=\frac1{\sqrt n}\sum_{j=\lceil na\rceil}^{ \lfloor nb\rfloor}\varepsilon(s_j)G^{(n)}(s_j)\\
	&=-v^*_n\frac1{\sqrt n}\sum_{j=\lceil na\rceil}^{ \lfloor nb\rfloor}\varepsilon(s_j)R(s_j)\log s_j+v_{0,n}\frac1{\sqrt n}\sum_{j=\lceil na\rceil}^{ \lfloor nb\rfloor}\varepsilon(s_j)R(s_j)\\
	&\ \ +2\sum_{k=1}^{\widetilde p}v_{k,n}\frac1{\sqrt n}\sum_{j=\lceil na\rceil}^{ \lfloor nb\rfloor}\varepsilon(s_j)R(s_j)\cos(2\pi ks_j).
	\end{align*}
	Each term in the last sum tends to zero, e.g., in the first term $v^*_n\to0$ and applying (\ref{Yconv}),
	in which $G_R(s_j)$ is replaced by $R(s_j)\log s_j$, 
	the sequence $\frac1{\sqrt n}\sum_{j=\lceil na\rceil}^{\lfloor nb\rfloor}\varepsilon(s_j)R(s_j)\log s_j$ has a weak limit.
	
\end{proof}
\bigskip\noindent
{\bf Acknowledgement.} This research was supported by the Ministry of Human Capacities, Hungary grant TUDFO/47138-1/2019-ITM.



\end{document}